\documentclass[12pt]{amsart}
\usepackage{amsmath, amssymb, amsthm, amsfonts, enumerate, verbatim}
\usepackage[all]{xy}
\usepackage{tikz}
\tikzstyle{punkt}=[circle, fill=black, minimum size=1mm,inner sep=0pt, draw]
\def\frk{\frak}               % font for "Fraktur"

\def\Phi{{\frk n}}
\def\Phi{{\frk N}}
\def\KK{{\mathbb K}}
\def\opn#1#2{\def#1{\operatorname{#2}}} % to make operators
\opn\chara{char}
\opn\length{\ell}
\opn\pd{pd}
\opn\rk{rk}
\opn\projdim{proj\,dim}
\opn\injdim{inj\,dim}
\opn\rank{rank}
\opn\depth{depth}
\opn\grade{grade}
\opn\height{height}
\opn\embdim{emb\,dim}
\opn\codim{codim}

\opn\Tr{Tr}
\opn\bigrank{big\,rank}
\opn\superheight{superheight}
\opn\lcm{lcm}
\opn\trdeg{tr\,deg}
\opn\reg{reg}
\opn\lreg{lreg}
\opn\ini{in}
\opn\lpd{lpd}
\opn\size{size}
\opn\bigsize{bigsize}
\opn\cosize{cosize}
\opn\bigcosize{bigcosize}
\opn\sdepth{sdepth}
\opn\sreg{sreg}
\opn\link{link}
\opn\fdepth{fdepth}
\opn\lin{lin}
\opn\ini{in}
\opn\div{div}
\opn\Div{Div}
\opn\cl{cl}
\opn\Cl{Cl}
\opn\Spec{Spec}
\opn\Supp{Supp}
\opn\supp{supp}
\opn\Sing{Sing}
\opn\Ass{Ass}
\opn\Min{Min}
\opn\Mon{Mon}
\opn\dstab{dstab}
\opn\astab{astab}
\opn\Syz{Syz}
\opn\Ann{Ann}
\opn\Rad{Rad}
\opn\Soc{Soc}
\opn\Im{Im}
\opn\Ker{Ker}
\opn\Coker{Coker}
\opn\Am{Am}
\opn\Hom{Hom}
\opn\Tor{Tor}
\opn\Ext{Ext}
\opn\End{End}
\opn\Aut{Aut}
\opn\id{id}

\opn\nat{nat}
\opn\pff{pf}%   \pf exists already
\opn\Pf{Pf}
\opn\GL{GL}
\opn\SL{SL}
\opn\mod{mod}
\opn\ord{ord}
\opn\Gin{Gin}
\opn\Hilb{Hilb}
\opn\sort{sort}
\opn\initial{init}
\opn\ende{end}
\opn\height{height}
\opn\type{type}
\opn\mdeg{mdeg}
\opn\aff{aff}
\opn\con{conv}
\opn\relint{relint}
\opn\st{st}
\opn\lk{lk}
\opn\cn{cn}
\opn\core{core}
\opn\vol{vol}
\opn\link{link}
\opn\star{star}
\opn\lex{lex}
\opn\sign{sign}
\opn\gr{gr}

\def\pot#1#2{#1[\kern-0.28ex[#2]\kern-0.28ex]}

\opn\dirlim{\underrightarrow{\lim}}
\opn\inivlim{\underleftarrow{\lim}}
\def\Implies{\ifmmode\Longrightarrow \else
        \unskip${}\Longrightarrow{}$\ignorespaces\fi}
\def\implies{\ifmmode\Rightarrow \else
        \unskip${}\Rightarrow{}$\ignorespaces\fi}
\def\iff{\ifmmode\Longleftrightarrow \else
        \unskip${}\Longleftrightarrow{}$\ignorespaces\fi}

\let\:=\colon
\newtheorem{Theorem}{Theorem}[section]
 \newtheorem{Lemma}[Theorem]{Lemma}
 \newtheorem{Corollary}[Theorem]{Corollary}
 \newtheorem{Proposition}[Theorem]{Proposition}

 \newtheorem{Conjecture}[Theorem]{Conjecture}

\let\epsilon\varepsilon
\let\kappa=\varkappa
\def\pnt{{\raise0.5mm\hbox{\large\bf.}}}

\begin{document}
\title{Binomial edge ideals of regularity~$3$}
\author {Sara Saeedi Madani and Dariush Kiani}

\address{Sara Saeedi Madani, Department of Mathematics and Computer Science, Amirkabir University of Technology (Tehran Polytechnic), Tehran, Iran}
\email{sarasaeedi@aut.ac.ir, sarasaeedim@gmail.com}

\address{Dariush Kiani, Department of Mathematics and Computer Science, Amirkabir University of Technology (Tehran Polytechnic), Tehran, Iran, and School of Mathematics, Institute for Research in Fundamental Sciences (IPM), Tehran, Iran}
\email{dkiani@aut.ac.ir, dkiani7@gmail.com}

\begin{abstract}
Let $J_G$ be the binomial edge ideal of a graph $G$. We characterize all graphs whose binomial edge ideals, as well as their initial ideals, have regularity $3$. Consequently we characterize all graphs $G$ such that $J_G$ is extremal Gorenstein. Indeed, these characterizations are consequences of an explicit formula we obtain for the regularity of the binomial edge ideal of the join product of two graphs. Finally, by using our regularity formula, we  discuss some open problems in the literature. In particular we disprove a conjecture in \cite{CDI} on the regularity of weakly closed graphs.  

%Given two graphs $G_1$ and $G_2$, we provide a formula describing the regularity of the binomial edge ideal of their join product in terms of the regularity of $J_{G_1}$ and $J_{G_2}$.  As another application of our formula, we show that for any integers $2\leq t\leq n$, there exists a graph $G$ on $n$ vertices with $\reg J_G=t$.      
\end{abstract}

\thanks{The research of the second author was in part supported by a grant from IPM (No. 95050116). }
%\thanks{}


\maketitle

\section{Introduction}\label{introduction}

Let $G$ be a finite simple graph, (i.e. with no loops, multiple or directed edges) on $n$ vertices and the edge set $E$.  Let $S=\KK[x_1,\ldots,x_n,y_1,\ldots,y_n]$ be the polynomial ring over a field $\KK$ with the indeterminates $x_1,\ldots,x_n,y_1,\ldots,y_n$, and let $f_{ij}:=x_iy_j-x_jy_i$ for $1\leq i<j\leq n$. Then the ideal $J_G$, generated by the binomials $f_{ij}$ in $S$ where $\{i,j\}\in E$, is known as the \emph{binomial edge ideal} of $G$, namely 
\[
J_G=(f_{ij}~:~i<j~,~\{i,j\}\in E).
\]   
Note that $J_G$ could be seen as the ideal generated by a collection of $2$-minors of the generic $(2\times n)$-matrix $X$. Therefore, the binomial edge ideal of a complete graph with $n$ vertices is just the determinantal ideal of $X$ which has been already studied very well. 

Binomial edge ideals were introduced in 2010 by Herzog, Hibi, Hreinsd\'ottir, Kahle and Rauh in \cite{HHHKR} and at about the same time by Ohtani in \cite{O}. In the meantime, it has been one of the most active areas of research, and there have been several research papers on this interesting class of binomial ideals, studying many of their algebraic properties and invariants. The reduced Gr\"obner basis and primary decomposition of these ideals as well as their minimal prime ideals were investigated in \cite{HHHKR}. In certain cases, characterizations for some properties like Cohen-Macaulay-ness and Gorenstein-ness were given, see for example \cite{BN,BMS,EHH,KS1,R,RR,Z,SZ}. 

One of the efforts in studying those ideals has been concerning their minimal graded free resolution. In \cite{SK}, all binomial edge ideals, as well as their initial ideals, with linear resolution were characterized. These ideals have in fact regularity~$2$, and the only graphs (without isolated vertices) which have this property are complete graphs. As a generalization of this result, all binomial edge ideals with pure resolutions were classified in \cite{KS}. The linear strand of $J_G$ was explicitly described via the so-called generalized Eagon-Northcott complex in \cite{HKS}. Some of the graded Betti numbers of $J_G$ were also studied in \cite{D,KS,SK,ZZ}. 

The Castelnuovo-Mumford regularity has been and still is an interesting invariant arising from the minimal graded free resolution of $J_G$. 
Some lower and upper bounds for the regularity of $J_G$ and $\ini_< J_G$ have been obtained in \cite{KS2} and \cite{MM}. Indeed, for a connected graph $G$ with $n$ vertices, the regularity of $J_G$ is bounded below by the length of the longest induced path of $G$ plus $1$, and it is bounded above by $n-1$. The authors posed a conjecture in \cite{SK1} which asserts that the regularity of $J_G$ is bounded above by $c(G)+1$ where $c(G)$ is the number of maximal cliques of $G$. This conjecture has been verified for closed graphs and block graphs in \cite{EZ}, where moreover an exact formula for the regularity of $J_G$ was obtained where $G$ is a closed graph. 

Finding explicit formulas for the regularity of $J_G$ for certain graphs $G$ is an interesting problem. On the other hand, as we mentioned above, the graphs for which $J_G$ has regularity $2$ were classified. So, a natural question could be if there are combinatorial characterizations for higher regularities. In this paper, we study these questions and discuss several related problems as applications. 

This paper is organized as follows. In Section~\ref{join}, we give an exact formula for the regularity of $J_{G_1*G_2}$, ($\ini_< J_{G_1*G_2}$, resp.) in terms of the regularities of $J_{G_1}$ and $J_{G_2}$, ($\ini_< J_{G_1}$ and 
$\ini_< J_{G_2}$, resp.). Here, by $G_1*G_2$ we mean the join (product) of two graphs $G_1$ and $G_2$ on disjoint sets of vertices, which is obtained from the union of $G_1$ and $G_2$ by joining all the vertices of $G_1$ to the vertices of $G_2$, (see Section~\ref{join} for precise definition). In Section~\ref{reg=3}, by applying the main theorem from Section~\ref{join} together with a graphical classification of $P_k$-free graphs, we obtain the full characterization of graphs $G$ for which $\reg (J_G)=3$. This characterization is a recursive construction based on join of some smaller graphs. Consequently, we deduce that $\reg (J_G)=3$ if and only if $\reg (\ini_< J_G)=3$ which gives a partial positive answer to a conjecture posed in \cite{HHHKR}. We would like to note that in spite of binomial edge ideals of regularity~$2$, those of regularity~$3$ arise from a large class of graphs. The family of Threshold graphs is an example of graphs with regularity~$3$. In this section, we also discuss higher regularities. Indeed, as an application of our main theorem in Section~\ref{join}, we show that for any positive integers $3\leq t\leq n$, there exists a connected graph with $n$ vertices whose binomial edge ideal has regularity~$t$. Finally, in Section~\ref{join}, we characterize the graphs $G$ such that $S/J_G$ is extremal Gorenstein. For this purpose, we first determine all graphs $G$ for which $\reg (J_G)=3$ and $S/J_G$ is Cohen-Macaulay. 
In Section~\ref{further}, we briefly discuss some interesting conjectures and questions concerning the regularity of binomial edge ideals. First by our regularity formula for the join of graphs, we disprove a conjecture by Chaudhry, Dokuyucu and Irfan in \cite{CDI} on the regularity of so-called weakly closed graph. This conjecture asserts that the regularity of $J_G$ where $G$ is a connected weakly closed graph is equal to $\ell(G)$, (i.e. the longest length of an induced path of $G$). Indeed, we give a family of connected counter-examples to this conjecture. Moreover, we show that the difference of $\reg (J_G)$ and $\ell (G)$ could be high enough, namely $\lim_{q\rightarrow \infty} \frac{\reg (J_{G_q})}{\ell(G_q)+1}=\infty$. 

We also discuss two other conjectures from \cite{EHH} by Ene, Herzog and Hibi, and from \cite{SK1} by the authors of this paper, respectively. More precisely, we show that those problems are join-closed, and hence we can extend the classes of graphs for which those problems have been solved, (see Section~\ref{further} for precise statements of those problems).  

In this paper, all graphs are finite simple graphs, and if a graph has $n$ vertices, then we sometimes use $[n]$ to denote its set of vertices.

\section{Regularity of binomial edge ideal of join product}\label{join}

Let $G_1$ and $G_2$ be two graphs on the set of vertices $V_1$ and $V_2$, and the edge sets $E_1$ and $E_2$, respectively. We denote by $G_1*G_2$, the \emph{join product} (or \emph{join}) of two graphs $G_1$ and $G_2$, that is
the graph with vertex set $V_1\cup V_2$, and the edge set
\[
E_1\cup E_2\cup \{\{v,w\}~:~v\in V_1,~w\in V_2\}.
\]

In this section we study the behavior of the regularity of the binomial edge ideal of the join of two graphs. Indeed, in the following theorem, we give a precise formula to compute the regularity of the binomial edge ideal of the join of two graphs, as well as its initial ideals, in terms of those of original graphs. More precisely, the following is the main result of this section. 
%Here ``$<$" denotes the lexicographic term order on $S$ induced by $x_1>\cdots >x_n>y_1>\cdots >y_n$.

\begin{Theorem}\label{reg-join}
	Let $G_1$ and $G_2$ be graphs on disjoint vertex sets $V_1$ and $V_2$, respectively, not both complete, and let $<$ be any term order on $S$. Then
	\begin{enumerate}
	   \item[{\em(a)}] $\mathrm{reg}(J_{G_1*G_2})=\mathrm{max}\{\mathrm{reg}(J_{G_1}),\mathrm{reg}(J_{G_2}),3\}$.
	   \item[{\em(b)}] $\mathrm{reg}(\ini_< J_{G_1*G_2})=\mathrm{max}\{\mathrm{reg}(\ini_< J_{G_1}),\mathrm{reg}(\ini_< J_{G_2}),3\}$.
	\end{enumerate}
\end{Theorem}

The above result shows that if $\reg (J_{G_1})$ and $\reg (J_{G_2})$ do not depend on the characteristic of the field $\KK$, then 
$\reg (J_{G_1*G_2})$ does not as well. The same holds for their initial ideals.  

Note that the join of two complete graphs is complete too, so that its binomial edge ideal has a linear resolution, by \cite[Theorem~2.1]{SK}. Hence, in this case, the regularity of the binomial edge ideal is equal to $2$.

Now, we need to fix some notation. If $H$ is a graph with connected components $H_1,\ldots,H_r$, then we denote it by $\bigsqcup_{i=1}^r H_i$. In particular, if all $H_i$'s are isomorphic to a graph $G$, then for simplicity we may write $rG$. 

Let $V$ be a set. The \emph{join} of two collections of subsets  $\mathcal{A}$ and $\mathcal{B}$ of $V$, denoted by $\mathcal{A}\circ \mathcal{B}$, was introduced in \cite{KS1} as
\[
\{A\cup B: A\in \mathcal{A}, B\in \mathcal{B}\}.
\]
The join of collections of subsets $\mathcal{A}_1,\ldots,\mathcal{A}_t$ of $V$ is denoted by $\bigcirc_{i=1}^{t}\mathcal{A}_i$.

We also need to recall a nice combinatorial description of the minimal prime ideals of binomial edge ideals given in \cite{HHHKR}. Let $G$ be a graph on $[n]$, and $T\subseteq [n]$, and let $G_1,\ldots,G_{c_G(T)}$ be the connected
components of $G_{[n]\setminus T}$, the induced subgraph of $G$ on $[n]\setminus T$. For any $i$, we denote by $\widetilde{G}_i$ the complete graph on the vertex set $V(G_i)$. Let
\[
P_T(G)=(\bigcup_{i\in T}\{x_i,y_i\}, J_{\widetilde{G}_1},\ldots,J_{\widetilde{G}_{c_{G}(T)}}),
\]
which is a prime ideal in $S$.
%where $\mathrm{height}\hspace{0.35mm}P_T(G)=n+|T|-c(T)$, by \cite[Lemma~3.1]{HHHKR}.
It was shown in \cite[Theorem~3.2]{HHHKR} that $J_G=\bigcap_{T\subset [n]}P_T(G)$. %So that, $\mathrm{dim}\hspace{0.35mm}S/J_G=\mathrm{max}\{n-|T|+c(T):T\subset [n]\}$, by \cite[Cororally~3.3]{HHHKR}.
A vertex whose removal from $G$ increases the number of connected components of $G$ is called a \emph{cut point} of $G$. If each $i\in T$ is a cut point of the graph $G_{([n]\setminus T)\cup \{i\}}$, then $T$ is said to have \emph{cut point property} for $G$. Now, let
\[
\mathcal{C}(G)=\{\emptyset\}\cup \{T\subset [n]:T~\mathrm{has~cut~point~property~for}~G\}.
\]
In particular, it is easy to see that $\mathcal{C}(G)=\{\emptyset\}$ if and only if $G$ is a complete graph. Furthermore, it was shown in \cite[Corollary~3.9]{HHHKR} that $T\in \mathcal{C}(G)$ if and only if $P_T(G)$ is a minimal prime ideal of $J_G$.

\medskip
To prove the main result of this section, we need to provide some ingredients. The next proposition from \cite{KS} describes the minimal prime ideals of $J_{G_1*G_2}$ when $G_1$ and $G_2$ are both disconnected.

\begin{Proposition}\label{both disconnected 1}
	\cite[Proposition~4.14]{KS}
	Let $G_1=\bigsqcup_{i=1}^r G_{1i}$ and $G_2=\bigsqcup_{i=1}^s G_{2i}$ be two graphs on disjoint sets of vertices
	$V_1=\bigcup_{i=1}^r V_{1i}$ and $V_2=\bigcup_{i=1}^s V_{2i}$,
	respectively, where $r,s\geq 2$. Then
	$$\mathcal{C}(G_1*G_2)=\{\emptyset\}\cup \big{(}(\bigcirc_{i=1}^{r}\mathcal{C}(G_{1i}))\circ \{V_2\} \big{)}\cup \big{(}(\bigcirc_{i=1}^{s}\mathcal{C}(G_{2i}))\circ \{V_1\} \big{)}.$$
\end{Proposition}

The next lemma follows with the same argument as in the proof of \cite[Lemma~3.1]{CDI} where $t=1$.

\begin{Lemma}\label{initial-variables}
	{\em(see} \cite[Lemma~3.1]{CDI}{\em)}
	Let $G$ be a graph on $[n]$ and $i_1,\ldots,i_t\in [n]$, and let $<$ be any term order on $S$. Then
	\[
	\ini_< (J_G,x_{i_1},y_{i_1},\ldots,x_{i_t},y_{i_t})=
	\ini_< J_G+(x_{i_1},y_{i_1},\ldots,x_{i_t},y_{i_t}).
	\]
\end{Lemma}

We also need the following lemma from \cite{C}.

\begin{Lemma}\label{Aldo}
	\cite[Lemma~1.3]{C}
	Let $R=\KK[z_1,\ldots,z_r]$ be a polynomial ring, and $I$ and $J$ be two homogeneous ideals in $S$. Let $<$ be any term order on $S$. Then the following statements are equivalent:
	\begin{enumerate}
		\item[{\em(a)}] $\ini_< (I+J)=\ini_< I+\ini_< J$;
		\item[{\em(b)}] $\ini_< (I\cap J)=\ini_< I\cap \ini_< J$.
	\end{enumerate}
\end{Lemma}

The following theorem determines exactly when a binomial edge ideal and its initial ideals have a linear resolution. This follows from \cite[Theorem~1.4]{B} and \cite[Theorem~2.1]{SK}.  

\begin{Theorem}\label{any order}
	{\em{(}see} \cite[Theorem~1.4]{B} {\em{and}} \cite[Theorem~2.1]{SK}{\em{)}}
Let $G$ be a graph with no isolated vertices, and let $<$ be any term order. Then the following conditions are equivalent:
\begin{enumerate}
	\item[{\em(a)}] $J_G$ has a linear resolution, {\em{(}}i.e. 
	$\reg (J_G)=2${\em{)}};
	\item[{\em(b)}] $\ini_< J_G$ has a linear resolution, {{\em{(}}}i.e. 
	$\reg (\ini_<J_G)=2${\em{)}};
	\item[{\em(c)}] $G$ is a complete graph. 
\end{enumerate}
\end{Theorem}

We would like to remark that the above theorem was proved in \cite[Theorem~2.1]{SK} for the lexicographic term order. The only part in the proof of that theorem which depends on the term order is that if $G$ is complete, then the desired initial ideal of $J_G$ is generated in degree $2$. But, since $G$ is complete, it follows that $J_G$ is just a determinantal ideal of a $(2\times n)$-generic matrix, and hence by \cite[Theorem~1.4]{B} all the graded Betti numbers of $J_G$ and any of its initial ideals coincide. In particular, all the initial ideals of $J_G$ are generated in degree $2$.      

\medskip
It was shown in \cite{SK1} that any induced subgraph of a graph $G$ provides an algebra retract for $S/J_G$. In particular, the following holds:
 
\begin{Proposition}\label{induced}
	\cite[Proposition~8]{SK1}
	Let $G$ be a graph, and let $H$ be an induced subgraph of $G$. Then $\reg (J_H)\leq \reg (J_G)$. 
\end{Proposition}

\medskip
Now we are ready to prove the main result of this section. \\

{\it Proof of Theorem~\ref{reg-join}:}
%\begin{enumerate}
%\item[(a)] 	
(a) Let $G:=G_1*G_2$. Since $G$ is not a complete graph, $J_G$ does not have a linear resolution, by Theorem~\ref{any order}. Therefore, $\mathrm{reg}(J_G)\geq 3$. By Proposition~\ref{induced}, $\mathrm{reg}(J_G)\geq\mathrm{reg}(J_{G_1})$ and $\mathrm{reg}(J_G)\geq\mathrm{reg}(J_{G_2})$, since $G_1$ and $G_2$ are both induced subgraphs of $G$. This implies that
\[
\mathrm{reg}(J_G)\geq \mathrm{max}\{\mathrm{reg}(J_{G_1}),\mathrm{reg}(J_{G_2}),3\}.
\]
To verify the inverse inequality, first we consider the case that both of $G_1$ and $G_2$ are disconnected graphs on $n_1$ and $n_2$ vertices, respectively. Let $G_1=\bigsqcup_{i=1}^r G_{1i}$ and $G_2=\bigsqcup_{i=1}^s G_{2i}$ with disjoint sets of vertices
	$V_1=\bigcup_{i=1}^r V_{1i}$ and $V_2=\bigcup_{i=1}^s V_{2i}$,
	respectively, where $r,s\geq 2$. Using Proposition~\ref{both disconnected 1},
	%$\mathcal{C}(G)=\{\emptyset\}\cup \big{(}(\bigcirc_{i=1}^{r}\mathcal{C}(G_{1i}))\circ \{[n_2]\} \big{)}\cup \big{(}(\bigcirc_{i=1}^{s}\mathcal{C}(G_{2i}))\circ \{[n_1]\} \big{)}$.
	one can decompose $J_G$ as $J_{G}=Q\cap Q'$, where
	\begin{equation}
	Q=\bigcap_{\substack{
			T\in \mathcal{C}(G) \\
			V_1\subseteq T
	}}
	P_T(G)~~,~~Q'=\bigcap_{\substack{
			T\in \mathcal{C}(G) \\
			V_1\nsubseteq T
	}}
	P_T(G).
	\nonumber
	\end{equation}
	Therefore,
	\begin{equation}
	Q=(x_i,y_i:i\in V_1)+\bigcap_{\substack{
			T\in \mathcal{C}(G) \\
			V_1\subseteq T
	}}
	P_{T\setminus V_1}(G_2)
	\nonumber
	\end{equation}
	and
	\begin{equation}
	Q'=P_{\emptyset}(G)\cap\big{(}\bigcap_{\substack{
			\emptyset\neq T\in \mathcal{C}(G) \\
			V_1\nsubseteq T
	}}
	P_T(G)\big{)}
	=P_{\emptyset}(G)\cap \big{(}(x_i,y_i:i\in V_2)+\bigcap_{\substack{
			T\in \mathcal{C}(G) \\
			V_2\subseteq T
	}}
	P_{T\setminus V_2}(G_1)\big{)}.
	\nonumber
	\end{equation}
	Then, it follows that
	\[
	Q=(x_i,y_i:i\in V_1)+J_{G_2}, \quad Q'=J_{K_n}\cap \big{(}(x_i,y_i:i\in V_2)+J_{G_1}\big{)},
	\]
	and hence
	\[
	Q+Q'=(x_i,y_i:i\in V_1)+J_{K_{n_2}}.
	\]
	Thus, we have $\mathrm{reg}(Q)=\mathrm{reg}(J_{G_2})$ and $\mathrm{reg}(Q+Q')+1=\mathrm{reg}(J_{K_{n_2}})+1=3$.
	Now, the short exact sequence
	\[
	0\rightarrow J_G\rightarrow Q\oplus Q'\rightarrow Q+Q'\rightarrow 0
	\]
	implies that
	\[
	\mathrm{reg}(J_G)\leq \mathrm{max}\{\mathrm{reg}(Q),\mathrm{reg}(Q'),\mathrm{reg}(Q+Q')+1\},
	\]
	by \cite[Corollary~18.7]{P}.	
	Similar to the above argument (by a suitable short exact sequence), it follows that
	\[
	\mathrm{reg}(Q')\leq
	\mathrm{max}\{\mathrm{reg}(J_{G_1}),\mathrm{reg}(J_{K_{n_1}})+1=3\},
	\]
	and so that $\mathrm{reg}(J_G)\leq \mathrm{max}\{\mathrm{reg}(J_{G_2}),\mathrm{reg}(J_{G_1}),3\}$.
	
	Next, assume that either $G_1$ or $G_2$ is connected. By adding  isolated vertices $v$ and $w$ to $G_1$ and $G_2$, respectively, two disconnected graphs $G_1'$ and $G_2'$ are obtained. Hence, by the previous case discussed above, it follows that
	\[
	\mathrm{reg}(J_{G_1'*G_2'})\leq \mathrm{max}\{\mathrm{reg}(J_{G_1'}),\mathrm{reg}(J_{G_2'}),3\}.
	\]
	Since $\mathrm{reg}(J_{G_1'})=\mathrm{reg}(J_{G_1})$ and $\mathrm{reg}(J_{G_2'})=\mathrm{reg}(J_{G_2})$, we get  
	\[
	\reg (J_G)\leq \mathrm{reg}(J_{G_1'*G_2'})\leq \mathrm{max}\{\mathrm{reg}(J_{G_1}),\mathrm{reg}(J_{G_2}),3\},
	\]
	where the first inequality for $J_G$ follows by Proposition~\ref{induced}, since $G=G_1*G_2$ is an induced subgraph of $G_1'*G_2'$. \\
	
	%\item[(b)]
	(b) We keep using the notation of the proof of part~(a). The ideal $\ini_< J_G$ does not have a linear resolution, by Theorem~\ref{any order}, since $G$ is not complete. Hence, $\mathrm{reg}(\ini_< J_G)\geq 3$. Let $S_p:=\KK[x_i,y_i:i\in V_p]$ for $p=1,2$. Then $(\ini_< J_G)\cap S_p=(\ini_< J_{G_p})S_p$ for $p=1,2$, which implies easily that $S_p/\ini_< J_{G_p}$ is an algebra retract of $S/\ini_< J_G$. Therefore, $\mathrm{reg}(\ini_< J_G)\geq\mathrm{reg}(\ini_< J_{G_p})$ for $p=1,2$, and hence we have 
	\[
	\mathrm{reg}(\ini_< J_G)\geq \mathrm{max}\{\mathrm{reg}(\ini_< J_{G_1}),\mathrm{reg}(\ini_< J_{G_2}),3\}.
	\]
	
	To prove the other inequality, first we assume that $G_1$ and $G_2$ are both disconnected. We show that
	\[
	\ini_<(Q\cap Q')=\ini_<Q\cap \ini_<Q'.
	\]
	For this, by Lemma~\ref{Aldo}, it is enough to show that \[
	\ini_<(Q+Q')=\ini_<Q+\ini_<Q'.
	\]
	It is clear that $\ini_<Q+\ini_<Q'\subseteq \ini_<(Q+Q')$, and we just need to verify the other inclusion. By Lemma~\ref{initial-variables}, we have
	\begin{equation}\label{Q+Q'}
		\ini_<(Q+Q')=(x_i,y_i:i\in V_1)+\ini_< J_{K_{n_2}}
	\end{equation}
	and
	\begin{equation}\label{Q}
	 \ini_<Q=(x_i,y_i:i\in V_1)+\ini_< J_{G_2}.
	\end{equation}
	Since $J_{K_{n_2}}\subseteq Q'$, one has
	$\ini_< J_{K_{n_2}}\subseteq \ini_<Q'$. This together with equations~(\ref{Q+Q'})~and~(\ref{Q}) implies that $\ini_<(Q+Q')\subseteq \ini_<Q+\ini_<Q'$.
	
	Therefore, we can consider the following short exact sequence:
	\[
		0\rightarrow \ini_<J_G=\ini_<Q\cap \ini_<Q'\rightarrow \ini_<Q\oplus \ini_<Q'\rightarrow \ini_<(Q+Q')\rightarrow 0.
	\]
	Therefore,
	\[
	\reg(\ini_<J_G)\leq \mathrm{max}\{\reg(\ini_<Q),\reg(\ini_<Q'),\reg(\ini_<(Q+Q'))+1\}.
	\]
	By equations~(\ref{Q+Q'})~and~(\ref{Q}),
	\[
	\reg (\ini_<(Q+Q'))=\reg (\ini_<J_{K_{n_2}})=2\quad \text{and}\quad
	\reg (\ini_< Q)=\reg (\ini_< J_{G_2}),
	\] respectively. Now, we need to compute $\reg (\ini_< Q')$, so that we use Lemma~\ref{Aldo} again. Namely, we have
	\[
	\ini_< Q'=
	\ini_<(J_{K_n})\cap \ini_<\big{(}(x_i,y_i:i\in V_2)+J_{G_1})\big{)},
	\]
	since
	\[
	\ini_<\big{(}J_{K_n}+((x_i,y_i:i\in V_2)+J_{G_1})\big{)}=
	\ini_<(J_{K_n})+ \ini_<((x_i,y_i:i\in V_2)+J_{G_1}),
	\]
	which is equal to $(x_i,y_i:i\in V_2)+\ini_< J_{K_n}$.
	Then, by taking a suitable short exact sequence (in a similar way as above), it follows that
	\[
	\reg (\ini_< Q')\leq \max \{\reg(\ini_<J_{G_1}),\reg (\ini_< J_{K_n})+1=3\}.
	\]
	Therefore, we get
	\[
	\mathrm{reg}(\ini_< J_{G_1*G_2})=\mathrm{max}\{\mathrm{reg}(\ini_< J_{G_1}),\mathrm{reg}(\ini_< J_{G_2}),3\}.
	\]
	Next, we assume that either $G_1$ or $G_2$ is connected. With the same method as in the proof of part~(a) of the theorem, we provide two new disconnected graphs $G_1'$ and $G_2'$. Then, by the above argument, we have
	\[
	\mathrm{reg}(\ini_< J_{G'_1*G'_2})=\mathrm{max}\{\mathrm{reg}(\ini_< J_{G'_1}),\mathrm{reg}(\ini_< J_{G'_2}),3\}.
	\]
	Let $S'$ be the polynomial ring with variables correspond to the vertices of the graph $G'_1*G'_2$. Then the desired result follows from the fact that $S/\ini_< J_{G_1*G_2}$ is an algebra retract of $S'/\ini_< J_{G'_1*G'_2}$, together with the fact that $\mathrm{reg}(\ini_< J_{G'_i})=\mathrm{reg}(\ini_< J_{G_i})$ for $i=1,2$.  
%\end{enumerate}
\qed

%We end this section with some remarks regarding Conjecture~A. \\

%Note that by \ref{reg-join}, we have if $G$ is a (multi)-fan graph (i.e. $K_1*\bigsqcup_{i=1}^t P_{n_i}$, for some $t\geq 1$, which might be a non-closed graph), then $\mathrm{reg}(J_G)=c(G)+1$. This implies that if Conjecture~A is true, then the given bound is sharp.

%\begin{Corollary}\label{ConjB-join}
	%Let $G_1$ and $G_2$ be two graphs on $[n_1]$ and $[n_2]$, respectively. If Conjecture~A is true for $G_1$ and $G_2$, then it is also true for $G_1*G_2$.
%\end{Corollary}

%\begin{proof}
	%By \ref{reg-join}, it is enough to note that $c(G_1*G_2)=c(G_1)c(G_2)$, and that if $G_1$ and $G_2$ are complete graphs, then $G_1*G_2$ is also complete and Conjecture~A is true for it.
%\end{proof}

\section{Characterization of binomial edge ideals with regularity $3$}\label{reg=3}

In this section our main goal is to characterize binomial edge ideals of regularity~$3$. As it was mentioned in the previous section, the binomial edge ideals with regularity $2$ were characterized in \cite{SK}. It is natural to ask about a combinatorial characterization of binomial edge ideals of higher regularities. In this section, as an application of the main result of Section~\ref{join}, we give such a characterization.  We need the next theorem from graph theory which gives a characterization of $P_k$-free graphs. For this, we first recall some necessary graph theoretical notion. 

By $P_k$ we mean the path graph with $k$ vertices, and by a $P_k$-free graph we mean a graph which has no induced subgraph isomorphic to $P_k$. A \emph{dominating set} of a graph $G$ is a subset $X$ of vertices of $G$ such that every vertex not in $X$ has a neighbor in $X$. A \emph{connected dominating set} of a graph $G$ is a dominating set $X$ for which the induced subgraph $G_X$ of $G$ is connected. A connected dominating set whose all proper subsets are not connected dominating sets is called a \emph{minimal connected dominating set}. 
A connected dominating set of minimum size is called a \emph{minimum connected dominating set}.

\begin{Theorem}\label{Pk-free}
	\cite[Theorem~4]{CS}
	Let $G$ be a connected $P_k$-free graph, with $k\geq 4$, and let $X$ be any minimum connected dominating set of $G$. Then $G_X$ is $P_{k-2}$-free or isomorphic to $P_{k-2}$.
\end{Theorem}

In the following, we denote the complete graph on $t$ vertices and its complementary graph by $K_t$ and $K_t^c$, respectively. Moreover, for a vertex $v$ of a graph $G$, denoted by $N(v)$ we mean the set of all adjacent vertices to $v$ in $G$. We also set $N[v]:=N(v)\cup \{v\}$.  

Note that whenever $J_G=(0)$, namely $G$ consists of isolated vertices, we have $\reg (S/J_G)=0$. In this case, we set $\reg (J_G)=-\infty$. 

\medskip 
The next theorem is the main application of Theorem~\ref{reg-join}:

%\begin{Theorem}\label{reg 3}
	%Let $G$ be a non-complete graph with $n$ vertices. Then 
	%$\reg (J_G)=3$ if and only if $G$ is one of the following graphs:
%	\begin{enumerate}
%		\item[{\em(a)}] $K_r \sqcup K_s \sqcup \mathrm{Iso}(t)$ with $r,s\geq 2$, $t\geq 0$ and $r+s+t=n$.
%		\item[{\em(b)}] $G=(G_1*G_2)\sqcup \mathrm{Iso}(t)$ where $t\geq 0$ and $G_i$ is a graph with $n_i<n$ vertices such that $n_1+n_2+t=n$ and $\reg (J_{G_i})\leq 3$ for $i=1,2$.  
%	\end{enumerate}
%\end{Theorem}

\begin{Theorem}\label{reg 3}
	Let $G$ be a non-complete graph with $n$ vertices and no isolated vertices. Then 
	$\reg (J_G)=3$ if and only if either 
	\begin{enumerate}
		\item[{\em(a)}] $G=K_r \sqcup K_s$ with $r,s\geq 2$ and $r+s=n$, or
		\item[{\em(b)}] $G=G_1*G_2$ where $G_i$ is a graph with $n_i<n$ vertices such that $n_1+n_2=n$ and $\reg (J_{G_i})\leq 3$ for $i=1,2$.  
	\end{enumerate}
\end{Theorem}

\begin{proof}
First, assume that $G$ is a disconnected graph with the connected components $H_1,\ldots, H_t$. Then $\reg (S/J_G)=\sum_{i=1}^t \reg(S/J_{H_i})$, and hence 
\[
\reg (J_G)=\sum_{i=1}^t \reg(J_{H_i})-t+1.
\] 
Since $G$ does not have any isolated vertices, we have $\reg(J_{H_i})\geq 2$ for all $i=1,\dots,t$. Therefore, it follows that $\reg (J_G)=3$ if and only if $t=2$ and $\reg(J_{H_1})=\reg(J_{H_2})=2$, since $t\geq 2$. Then, Theorem~\ref{any order} implies that $\reg (J_G)=3$ if and only if $G=K_r \sqcup K_s$ with $r,s\geq 2$ and $r+s=n$. 

Next, suppose that $G$ is connected. If $G=G_1*G_2$ satisfies condition~(b), then it immediately follows from Theorem~\ref{reg-join} that $\reg (J_G)=3$. Now, we prove the converse. Suppose that $\reg (J_G)=3$. Then, by Proposition~\ref{induced}, $G$ is $P_4$-free, since $\reg (J_{P_4})=4$. Thus, by Theorem~\ref{Pk-free}, for any minimum connected dominating set $X$ of $G$, $G_X$ is $P_2$-free or isomorphic to $P_2$. Hence, for any minimum connected dominating set $X$ of $G$, $G_X$ is just a vertex or it is just an edge. Note that $G$ has a connected dominating set, because it is connected. Now, let $X$ be a minimum connected dominating set of $G$. If $G_X$ is a vertex, then it follows that $G=K_1*G_{V\setminus \{v\}}$. Since $G_{V\setminus \{v\}}$ is an induced subgraph of $G$, we have $\reg (J_{G_{V\setminus \{v\}}})\leq 3$, and hence $G$ satisfy condition~(b). Now assume that  $G_X$ is an edge, say $\{u,w\}$. Let $U:=N(u)\setminus N[w]$, $W:=N(w)\setminus N[u]$ and $Z:=N(u)\cap N(w)$. If $U=\emptyset$ or $W=\emptyset$, then it follows that $G=K_1*G_{V\setminus w}$ or $G=K_1*G_{V\setminus u}$, and hence the result follows similar to the previous case. Therefore, we assume that $U$ and $W$ are both non-empty. Note that all the vertices in $U$ are adjecent to all the vertices in $W$, since otherwise passing through $u$ and $w$, an induced path isomorphic to $P_4$ exists, which is a contradiction. On the other hand, if there is a vertex $z\in Z$ and vertices in $u_1\in U$ and $w_1\in W$ such that $z$ is not adjacent to neither $u_1$ nor $w_1$, then $z,u,u_1,w_1$ provides an induced path in $G$. But this is a contradiction, because $G$ is $P_4$-free. Therefore, any vertex in $Z$ is adjacent either to all the vertices in $U$ or to all the vertices in $W$. Without loss of generality, we assume that all the vertices of $Z$ and $U$ are adjacent. Then, it follows that $G=G_{Z\cup W\cup \{u\}}*G_{U\cup \{w\}}$. Since both of $G_{Z\cup W\cup \{u\}}$ and $G_{U\cup \{w\}}$ are induced subgraphs of $G$, the regularity of their binomial edge ideals is 
$\leq 3$, and hence $G$ satisfies condition~(b). Therefore, we get the  desired result.                   
\end{proof}

The above theorem recovers \cite[Proposition~3.1]{Heroli} in the special case of closed graphs. Note that by Theorem~\ref{reg 3}, the graphs $K_r \sqcup K_s$ with $r,s\geq 2$ are the only disconnected graphs without isolated vertices whose binomial edge ideals have regularity~$3$. Moreover, either $G_1$ or $G_2$ in part~(b) of Theorem~\ref{reg 3}, could be disconnected, even just a bunch of isolated vertices. Also, note that it is clear that the regularity of the binomial edge ideal of a graph is not changed by adding isolated vertices. 
%\end{Remark}

\medskip 
The graphs $G$ with $\reg (J_G)=3$ which were described in Theorem~\ref{reg 3} can be constructed recursively. Indeed, given a positive $n$, let 
\[
\mathcal{G}(n):=\{G:|V(G)|=n, \reg(J_G)\leq 3\}.
\]
First note that, by Theorem~\ref{any order}, the only graphs $G$ in $\mathcal{G}(n)$ with $\reg (J_G)<3$ are of the form $K_r \sqcup K_t^c$ with $r+t=n$ and $r\geq 1$ and $t\geq 0$. Now, let $G\in \mathcal{G}(n)$. By Theorem~\ref{reg 3}, it follows that either $G=K_r \sqcup K_s\sqcup K_t^c$ with $r,s\geq 2$ and $r+s+t=n$ for some $t\geq 0$, or $G=(G_1*G_2)\sqcup K_t^c$ where $t\geq 0$ and $G_i\in \mathcal{G}(n_i)$ for $i=1,2$ such that $n_i<n$ and $n_1+n_2+t=n$. In the latter case, since $G_i\in \mathcal{G}(n_i)$, one can apply again Theorem~\ref{reg 3} for $G_1$ and $G_2$. By proceeding in this way, after a finite number of steps, we obtain graphs of the form $K_r \sqcup K_s\sqcup K_t^c$ where $r\geq 1$ and $s,t\geq 0$. So, roughly speaking, these graphs could be seen as building blocks of the elements of $\mathcal{G}(n)$. 

It is clear that the smallest graph whose binomial edge ideal has regularity~$3$ is $P_3$, and it is in fact the only graph on three vertices with this property. Now, as an example, let us apply the above construction for $n=4$. Indeed, applying the above construction implies that the only graphs $G$ on four vertices with $\reg (J_G)=3$ are the following six graphs: $2K_2$, $K_2^c*K_2^c$ (i.e. the $4$-cycle), $K_2^c*K_2$, $(K_2\sqcup K_1)*K_1$, $K_3^c*K_1$ (i.e. the star graph $K_{1,3}$), and $K_1\sqcup (K_1*K_2^c)$.    
 
Note that \emph{threshold graphs}, a well-known large class of graphs, provide a special class of graphs constructed as above. We are grateful to Asghar~Bahmani who pointed out this nice class of graphs to us.  

\medskip 
The next corollary gives a partial positive answer to a conjecture, due to Ene, Herzog and Hibi, posed in \cite[page~68]{EHH}. This conjecture says that the extremal Betti numbers of $J_G$ and $\ini_{<_{lex}}J_G$, and in particular their regularities, coincide. Recall that Theorem~\ref{any order} implies that $\reg (J_G)=2$ if and only if $\reg (\ini_< J_G)=2$ for any term order $<$. In the next section, we make further comments on this conjecture.  

\begin{Corollary}\label{initial reg 3}
	Let $G$ be a graph, and let $<$ be any term order on $S$. Then $\reg (J_G)=3$ if and only if $\reg (\ini_< J_G)=3$. In particular, if $\reg (J_G)=3$, then $\ini_< J_G$ is generated in degree at most~$3$.  
\end{Corollary}

\begin{proof}
	If $\reg (\ini_< J_G)=3$, then $\reg (J_G)\leq 3$, by 
	\cite[Corollary~3.3.4]{HH}. Hence, we deduce that $\reg (J_G)=3$, by Theorem~\ref{any order}. Conversely, suppose that $\reg (J_G)=3$. We may assume that $G$ has no isolated vertices. If $G=K_r\sqcup K_s$, for $r,s\geq 2$, then $\ini_< J_G=\ini_< J_{K_r}+\ini_< J_{K_s}$, and hence $\reg (\ini_< J_G)=3$, since $K_r$ and $K_s$ are on disjoint sets of vertices, and $\reg (\ini_< J_{K_r})=\reg (\ini_< J_{K_s})=2$. Now, let $G$ satisfy condition~(b) of Theorem~\ref{reg 3}. Then it is enough to use induction on the number of vertices. Hence, the result follows from Theorem~\ref{reg-join} and Theorem~\ref{reg 3}. 
\end{proof}

Note that by Theorem~\ref{reg 3} the characterization of binomial edge ideals of regularity~$3$, as well as regularity~$2$, is indeed independent of the characteristic of the field $\KK$. 

\medskip
%In this section we discuss some other applications of Theorem~\ref{reg-join}. We divide this section into three different subsections. 
%\subsection{Binomial egde ideals with higher regularities}\label{higher}
After having characterized binomial edge ideals with regularity~$2$~and~$3$, now it is natural to ask about higher regularities. In general, the regularity of $J_G$ is bounded above by $n$ and this bound is attained if and only if $G$ is the path graph on $n$ vertices, see \cite[Theorem~3.2]{KS2} and \cite[Theorem~1.1]{MM}. Now, one may ask the following question: {\em Given positive integers $n$ and $t$ with $3\leq t<n$, does there exist a connected graph $G$ with $n$ vertices and $\reg (J_G)=t$?}

Applying Theorem~\ref{reg-join}~(a), it follows that the above question has a positive answer. Indeed, let $G_1$ be any graph with $n_1$ vertices and $\reg (J_{G_1})=t$, and let $G_2$ be any graph with $n_2$ vertices and $\reg (J_{G_2})\leq t$, where $n_1,n_2>0$ and $n_1+n_2=n$. Then, $G_1*G_2$ has $n$ vertices and Theorem~\ref{reg 3} implies that $\reg (J_{G_1*G_2})=t$. In particular, $G_1$ could be $P_t$ and $G_2$ could be any graph with $n-t$ vertices with $\reg (J_{G_2})\leq t$, for example $G_2$ could be just $K_{n-t}^c$. 

%Applying Theorem~\ref{reg-join}~(b) shows that the same fact holds for the initials of   

%\subsection{Extremal Gorenstein binomial edge ideals}\label{extremal}

\medskip 
In the rest of this section, our aim is to characterize all binomial edge ideals which are extremal Gorenstein, as a consequence of Theorem~\ref{reg 3}. First we recall the definition. Note that if $I$ is a graded ideal in a polynomial ring $R$ over a field such that $R/I$ is Gorenstein, then $I$ can never have a linear resolution unless $I$ is a principal ideal. However, if the minimal graded free resolution of $I$ is as linear as possible, then $I$ is said to be 
\emph{extremal Gorenstein}. In particular, in the case that $I$ is a graded ideal generated in degree $2$, it is extremal Gorenstein if $R/I$ is Gorenstein and $\reg (I)=3$.    

Now, using Theorem~\ref{reg-join}~(a), we characterize all Cohen-Macaulay binomial edge ideals with regularity $3$. Note that all  binomial edge ideals with regularity~$2$ are Cohen-Macaulay, since by Theorem~\ref{any order}, they are just determinantal ideals. In the proof of the following proposition, we use the classical notion of $\ell$-connected-ness of a graph; given a positive integer $\ell$, a connected graph $G$ on at least $\ell + 1$ vertices  is called $\ell$-\emph{connected}, if the induced subgraph obtained by deleting any subset of vertices of cardinality less than $\ell$ from $G$ is a connected graph as well. It follows obviously from the definition that if $G$ is an $(\ell + 1)$-connected graph, then it is also 
$\ell$-connected.         

\begin{Proposition}\label{CM reg 3}
Let $G$ be a graph on the vertex set $[n]$ which has no isolated vertices. Then the following statements are equivalent:
	\begin{enumerate}
		\item[{\em(a)}] $S/J_G$ is Cohen-Macaulay and $\reg (J_G)=3$;
		\item[{\em(b)}] $G=K_r\sqcup K_s$ with $r,s\geq 2$ and $r+s=n$, or $G=K_1*(K_r\sqcup K_s)$ with $r,s\geq 1$ and $r+s=n-1$. 
	\end{enumerate}
\end{Proposition}   

\begin{proof}
If $G=K_r\sqcup K_s$ with $r,s\geq 2$ and $r+s=n$, then conditions in (a) follows, by the fact that $J_{K_r}$ and $J_{K_s}$ are determinantal ideals and by Theorem~\ref{reg 3}. If $G=K_1*(K_r\sqcup K_s)$ with $r,s\geq 1$ and $r+s=n-1$, then $\reg (J_G)=3$ by Theorem~\ref{reg 3}. In this case, $G$ is clearly a block graph (with blocks $K_{r+1}$ and $K_{s+1}$). Then by \cite[Theorem~1.1]{EHH}, it follows that $S/J_G$ is Cohen-Macaulay. 

Conversely, assume that $S/J_G$ is Cohen-Macaulay and 
$\reg (J_G)=3$. Therefore, $G$ is not a complete graph by Theoem~\ref{any order}. If $G$ is disconnected, then Theorem~\ref{reg 3} implies that $G=K_r\sqcup K_s$ for some $r,s\geq 2$ such that $r+s=n$. Now suppose that $G$ is connected. Then, by Theorem~\ref{reg 3} it follows that $G=G_1*G_2$ with $\reg (J_{G_1}),\reg (J_{G_2})\leq 3$, where $G_1$ and $G_2$ are graphs on $n_1$ and $n_2$ vertices, respectively, such that $n_1+n_2=n$ and $1\leq n_1,n_2<n$. Without loss of generality, we assume that $n_1\leq n_2$. Then it can be easily seen that $G$ is $n_1$-connected. Therefore, by \cite[Proposition~3.10]{BN}, we deduce that $n_1=1$ and hence $G_1=K_1$, since $S/J_G$ is Cohen-Macaulay. This implies that $n_2\geq 2$, because $G$ is not complete. If $G_2$ is connected, then it follows that  $G=G_1*G_2$ is $2$-connected. Thus, $G_2$ has to be disconnected, and hence $G_2=K_r\sqcup K_s\sqcup K_t^c$ where $r,s\geq 1$ and $t\geq 0$, by Theorem~\ref{any order} and Theorem~\ref{reg 3}. Therefore, $G$ is a block graph (with blocks $K_{r+1}$, $K_{s+1}$, and $t$ copies of $K_2$ if $t\geq 1$). Then, again by using \cite[Theorem~1.1]{EHH}, we deduce that $t=0$, or equivalently $G$ has only two blocks $K_{r+1}$ and $K_{s+1}$, because $S/J_G$ is Cohen-Macaulay. Therefore, $G=K_1*(K_r\sqcup K_s)$ with $r,s\geq 1$ and $r+s=n-1$.              	
\end{proof}

Finally, we give the characterization of extremal Gorenstein binomial edge ideals. Indeed, binomial edge ideals rarely admit this property. 

\begin{Corollary}\label{Gor reg 3}
Let $G$ be a graph with no isolated vertices. Then the following statements are equivalent: 
\begin{enumerate}
	\item[{\em(a)}] $J_G$ is extremal Gorenstein;
	\item[{\em(b)}] $G=2K_2$ or $G=P_3$. 
\end{enumerate}
\end{Corollary}

\begin{proof}
It is enough to determine for which graphs $G$ of the forms in part~(b) of Proposition~\ref{CM reg 3}, $S/J_G$ is Gorenstein. 
If $G=K_r\sqcup K_s$ with $r,s\geq 2$, then $S/J_G$ is Gorenstein if and only if $r=s=2$. Indeed, the minimal graded free resolution of $S/J_{K_r\sqcup K_s}$ is just the tensor product of the ones of $S'/J_{K_r}$ and $S''/J_{K_r}$ where $S'$ and $S''$ are suitable polynomial rings over $\KK$, and $S'/J_{K_r}$ and $S''/J_{K_s}$ for $r,s\geq 2$ are resolved by the Eagon-Northcott complex. This implies that the last Betti number of $J_{K_r\sqcup K_s}$ for $r,s\geq 2$ is equal to~$1$ if and only if $r=s=2$.  
If $G=K_1*(K_r\sqcup K_s)$ with $r,s\geq 1$, then $G$ is a closed graph with maximal cliques $K_{r+1}$ and $K_{s+1}$ which intersect in one vertex. Then by \cite[Corollary~3.4]{EHH}, it follows that $S/J_G$ is Gorenstein if and only if $r=s=1$, and hence $G=P_3$. 
%If $G=2K_2$, then the last Betti number of $J_G$ is equal to~$1$, since $J_{K_2}$ is just a principal ideal. Moreover, we know that $J_{P_3}$ has regularity~$3$, and hence by \cite[Corollary~3.4]{EHH} it is extremal Gorenstein. 
%Conversely, let $G$ be a graph with no isolated vertices for which $J_G$ is extremal Gorenstein. If $G$ is disconnected, then by Proposition~\ref{CM reg 3}, it follows that $G=2K_2$, because $J_{K_r}$ for $r\geq 3$ is resolved with the Eagon-Northcott complex, and hence its last Betti number is bigger than one which implies that the last Betti number of $J_{K_r\sqcup K_s}$ for $r,s\geq 3$ is also bigger than one. Note that the minimal graded free resolution of $S/J_{K_r\sqcup K_s}$ is just the tensor product of the ones of $S/J_{K_r}$ and $S/J_{K_s}$. Now, assume that $G$ is connected. Then by Proposition~\ref{CM reg 3}, $G=K_1*(K_r\sqcup K_s)$ with $r,s\geq 1$.     
\end{proof}

\section{Further remarks on related problems}\label{further}

%\section{Join-closed conjectures in binomial edge ideals}\label{conjectures}

In this section we would like to make some remarks on some nice conjectures and questions concerning the regularity of binomial edge ideals. Indeed, by Theorem~\ref{reg-join} we get some negative and positive answers to those conjectures. 

We start with disproving a conjecture by Chaudhry, Dokuyucu and Irfan in \cite{CDI}. Let us first recall the notion of \emph{weakly closed} graphs, as a generalization of closed graphs. Let $G$ be a graph with the vertex set $[n]$ and edge set $E$, which admits a labeling of vertices with the following property: for all $i,j\in [n]$ with $j>i+1$ and $\{i,j\}\in E$, and for all $i<k<j$, one has either 
$\{i,k\}\in E$ or $\{j,k\}\in E$. Then $G$ is called a weakly closed graph. Weakly closed graphs were introduced in \cite{M}. Any closed graph is weakly closed. 

For any graph $G$, let $\ell(G)$ denote the length of the longest induced path in $G$. As we recalled in the introduction, if $G$ is a connected graph, then $\reg (J_G)\geq \ell(G) +1$, by Proposition~\ref{induced}. The following conjecture appeared in \cite{CDI}:

\begin{Conjecture}\label{weakly closed}
	If $G$ is a connected weakly closed graph, then 
	$\reg(J_G)=\ell(G) +1$. 
\end{Conjecture}  

Now, we give a family of counter-examples to this conjecture. First note that it is easily observed that the join of two weakly closed graphs $G_1$ and $G_2$ is a connected weakly closed graph. On the other hand, it is clear that $\ell(G)=\max \{\ell(G_1),\ell(G_2)\}$. Now, using these facts together with Theorem~\ref{reg-join}, one can construct various counter-examples. Here, we give the following family of graphs as an example: Let $H_1=K_1$, and let $q\geq 2$ be a positive integer and $H_2=\bigsqcup_{i=1}^q P_{t_i}$, (i.e. a disjoint union of $q$ paths), with $3\leq t_1\leq \cdots \leq t_q$. Since any path graph is a closed graph, it is also weakly closed. Therefore, as we mentioned above, $G=H_1*H_2$ is weakly closed and $\ell(G)=t_q-1$. However, Theorem~\ref{reg-join}~(a) implies that 
$\reg (J_{G})=\sum_{i=1}^{q}t_i-q+1$, which yields that  
$\reg (J_{G})>t_q-1$.  Furthermore, by constructing a sequence of graphs 
$(G_q)_{q\geq 2}$, as above, where 
$G_q=K_1*(\bigsqcup_{i=1}^q P_{t_i})$ and $t_1=\cdots =t_q=q^2$, it follows that 
\[
\lim_{q\rightarrow \infty} \frac{\reg (J_{G_q})}{\ell(G_q)+1}=\infty.
\]      
This shows that the difference between $\reg (J_{G})$ and $\ell (G)$ could be big enough in connected graphs.  

\medskip 
Besides Conjecture~\ref{weakly closed}, the more general problem that for which graphs the equality to $\ell(G) +1$ is attained, was discussed in \cite{CDI}. This is indeed a reasonable question. By Theorem~\ref{reg-join}~(a), it follows that the class of such graphs is \emph{join-closed}, namely if the regularities of $J_{G_1}$ and $J_{G_2}$ attain the lower bounds $\ell(G_1)$ and $\ell (G_2)$, respectively, then $J_{G_1*G_2}$ do. This follows since the longest induced path of $G_1*G_2$ is equal to the maximum of the $G_1$ and $G_2$, as we mentioned above. In \cite{EZ}, it was shown that connected closed graphs belong to this class. Moreover, in \cite{CDI}, the so-called $C_{\ell}$ graphs were given as another examples of such graphs.

Next, we deal with the conjecture which we pointed out in Section~\ref{reg=3}. This conjecture is due to Ene, Herzog and Hibi. Here, we mention a partial form of this conjecture:

\begin{Conjecture}\label{conj-extremal betti}
	\cite[Page~68]{EHH}
	Let $G$ be a graph. Then $\reg (J_G)=\reg (\ini_{<_{lex}}J_G)$.  
\end{Conjecture} 

Theorem~\ref{reg-join} indicates that if the above conjecture holds for $G_1$ and $G_2$, then it holds for the graph $G_1*G_2$, as well. Therefore, the class of graphs for which Conjecture~\ref{conj-extremal betti} holds, is also join-closed. Among the known graphs in this class, are the closed graphs and $C_{\ell}$-graphs, (see \cite{CDI} and \cite{EZ}). Note that the join of two graphs of aforementioned forms is not necessarily of the same type.   

\medskip 
The next conjecture we would like to look at, is the following due to the authors:

\begin{Conjecture}\label{conj-ours}
	\cite[Page~12]{SK1}
	Let $G$ be a graph, and let $c(G)$ be the number of maximal cliques of $G$. Then $\reg (J_G)\leq c(G)+1$. 
\end{Conjecture}

The class of graphs for which Conjecture~\ref{conj-ours} holds, is also join-closed. Indeed, by Theorem~\ref{reg-join}~(a), if $G_1$ and $G_2$ belong to this class, then $G_1*G_2$ belongs too, since it is easily seen that $c(G_1*G_2)=c(G_1)c(G_2)$. From the graphs of this class we can mention closed graphs and block graphs, (see \cite{EZ}), as well as graphs with no triangles where $c(G)$ equals the number of edges. It is clear that the join of these types of graphs could provide graphs not with the same type.

%\medskip 
%We would like to close this section by the following counterexample to a conjecture posed in \cite{CDI}. The conjecture says that if $G$ is a connected \emph{weakly closed} graph with, then $\reg (J_G)=\ell +1$ where $\ell$ is the length of the longest induced path of $G$. Recall that a graph is weakly closed if ...........

%Let $G$ be the graph with the vertex set $[6]$ and edges 
%$\{1,2\}$, $\{1,3\}$, $\{2,3\}$, $\{1,4\}$, $\{2,5\}$ and $\{3,6\}$. Then $G$ is a weakly closed graph whose longest induced path has length~$3$, while computations with \textit{CoCoA} show that 
%$\reg (J_G)=......$.   

\end{document}